\documentclass{amsart}
\usepackage{amsmath,amsthm,amssymb,color}
\usepackage{graphicx}
\newtheorem{Theorem}{Theorem}[section]

\newtheorem{Lemma}[Theorem]{Lemma}
\newtheorem{Proposition}[Theorem]{Proposition}
\newtheorem{Remark}{Remark}[section]
\begin{document}
%[[[ Prescript
%[[[ Title
\title[Self-similar solutions to DNLS]%
{Self-similar solutions to the derivative nonlinear Schr\"odinger equation}
%]]]

%[[[ Fujiwara
\author[K. Fujiwara]{Kazumasa Fujiwara}

\address[K. Fujiwara]{
Centro di Ricerca Matematica Ennio De Giorgi\\
Scuola Normale Superiore\\
Collegio Puteano, \\
Pisa, I-56100  \\ Italy}

\email{kazumasa.fujiwara@sns.it}
%]]]

%[[[ V. Georgiev----------Author 1
\author[V. Georgiev]{Vladimir Georgiev}

\address[V.Georgiev]{
Department of Mathematics,
University of Pisa,
Largo Bruno Pontecorvo 5,
I - 56127 Pisa, Italy}
\address{
Faculty of Science and Engineering, Waseda University,
3-4-1, Okubo, Shinjuku-ku, Tokyo 169-8555, Japan}
\address{
IMI--BAS, Acad. Georgi Bonchev Str.,
Block 8, 1113 Sofia, Bulgaria
}

\email{georgiev@dm.unipi.it}

\thanks{The second author was supported in part by  Project 2017 "Problemi stazionari e di evoluzione nelle equazioni di campo nonlineari" of INDAM,
GNAMPA - Gruppo Nazionale per l'Analisi Matematica,
la Probabilit\`a e le loro Applicazioni,
by Institute of Mathematics and Informatics,
Bulgarian Academy of Sciences and Top Global University Project, Waseda University,  by the University of Pisa, Project PRA 2018 49 and project "Dinamica di equazioni nonlineari dispersive", "Fondazione di Sardegna" ,2016.
ORCID:https://orcid.org/0000-0001-6796-7644}
%]]]

%[[[ T. Ozawa
 \author[T. Ozawa]{Tohru Ozawa}

\address[T. Ozawa]{T. Ozawa,
Department of Applied Physics \\ Waseda University \\
3-4-1, Okubo, Shinjuku-ku, Tokyo 169-8555 \\ Japan}

\email{txozawa@waseda.jp}
\thanks{The third author was supported by
Grant-in-Aid for Scientific Research (A) Number 26247014.}
%]]]

%[[[ Abstract
\begin{abstract}
A class of self-similar solutions
to the derivative nonlinear Schr\"odinger equations is studied.
Especially, the asymptotics of profile functions are shown
to posses a logarithmic phase correction.
This logarithmic phase correction is obtained from
the nonlinear interaction of profile functions.
This is a remarkable difference from the pseudo-conformally invariant case,
where the logarithmic correction comes from the linear part of the equations
of the profile functions.
\end{abstract}
%]]]

%[[[ Other data
% \date{\today}
%]]]

%]]]

\maketitle

%[[[ section : Introduction
\section{Introduction}
%[[[ In this paper,
Among the evolution equations
which have been derived from the magnetohydrodynamics equations,
the derivative nonlinear Schr\"odinger equation (DNLS) has attracted
most attention (see \cite{CLP99,HO92,HO,H,SLPS10,W1})
	\begin{align}
	\begin{cases}
	i \partial_t u + \partial_x^2 u + i \partial_x (|u|^2u) = 0,
	&分\qquad t \geq 1, \quad x \in \mathbb R,\\
	u(1,x) = u_0(x),
	&\qquad x \in \mathbb R.
	\end{cases}
	\label{eq:1}
	\end{align}
%]]]

%[[[ In $L^2$ or $H^1$ frame work,
DNLS is known to have infinite number of conservation laws
and it is completely integrable.
In $L^2$ or $H^1$ framework,
the equation has the following basic conservation quantities
	\begin{align*}
	\mbox{Mass:\quad} M(u)(t) &= \int_{\mathbb{R}} |u(t,x)|^2 dx,\\
	\mbox{Energy:\quad} E(u)(t) &= \int_{\mathbb{R}} |\partial_x u(t,x)|^2 dx
	+ \frac{3}{2}
	\mathrm{Im} \int_{\mathbb{R}} |u(t,x)|^2 u(t,x) \overline{\partial_x u(t,x)} dx\\
	&+ \frac{1}{2} \int_{\mathbb{R}} | u(t,x)|^6 dx.
	\end{align*}
It is well known (see \cite{HO})
that DNLS \eqref{eq:1} has a global solution in $H^1(\mathbb{R})$,
provided that mass is less than $2\pi$.
The constant $2\pi$ is improved to $4\pi$ in \cite{W1,W2}.
%]]]

%[[[ A simple gauge transform
A simple gauge transform
	\[
	u \mapsto v(t,x)
	= u(t,x) \exp \bigg( \frac{i}{2} \int_{-\infty}^x |u(t,y)|^2 dy \bigg)
	\]
transforms DNLS in \eqref{eq:1} into
	\[
	i \partial_t v + \partial_x^2 v + i |v|^{2} \partial_x v
	= 0.
	\]
For details, see \cite{HO92,HO}.
Then, one can consider the generalized DNLS
	\begin{align}
	i \partial_t v + \partial_x^2 v
	+ i \varepsilon |v|^{2\sigma} \partial_x v
	=0,
	\label{eq:2}
	\end{align}
for $\varepsilon = \pm 1$ and $\sigma \geq 1$.
The existence and asymptotic properties of ``quasi'' self-similar solutions
to the generalized DNLS with $\varepsilon = 1$ and $\sigma > 1$
have been studied in \cite{CSS17}.
Here, quasi self-similar solutions are solutions to \eqref{eq:2} of the form
	\[
	u(t,x)
	= C_1 (T^\ast -t)^{-1/4 \sigma} Q\bigg( \frac{x-x^\ast}{\sqrt{T^\ast - t}} \bigg)
	e^{i C_2 + i C_3 \log(T^\ast-t)}
	\]
with some constants $C_1$, $C_2$, and $C_3$, and a complex valued function $Q$
satisfying a second order ordinary differential equation (ODE).
The motivation for their analysis combined with numerical calculations
was to predict the profile of the blowup of the solution in the case where $\sigma >1$.
For related subjects,
we also refer the reader to \cite{HO16}.
%]]]

%[[[ Our main goal
Our main goal in  this work is
to prove the existence of self-similar solutions to DNLS \eqref{eq:1},
of the form
	\begin{align}
	u(t,x) = t^{-1/4} Q (t^{-1/2} x),
	\label{eq:3}
	\end{align}
and to derive precise asymptotic expansions for the profile function $Q$.
%]]]

%[[[ For pseudo-conformally invariant nonlinear Schr\"odinger equations,
For pseudo-conformally invariant nonlinear Schr\"odinger equations,
self-similar solutions taking the form \eqref{eq:3}
were studied by Kavian and Weissler \cite{KW94}.
Particularly,
they considered the following Cauchy problem
	\begin{align}
	\begin{cases}
	i \partial_t u + \Delta u + |u|^{4/n}u = 0,
	&\qquad t \geq 1, \quad x \in \mathbb R^n,\\
	u(1,x) = u_0(x),
	&\qquad x \in \mathbb R^n
	\end{cases}
	\label{eq:4}
	\end{align}
and obtained solutions of the form
	\[
	u(t)
	= t^{i \omega_0 / 2 - n /(4-n) } Q_{\mathrm{NLS}}(t^{-1/2} x),
	\]
where $\omega_0$ is a constant
and $Q_{\mathrm{NLS}}$ is a complex valued function satisfying
	\begin{align}
	\Delta Q_{\mathrm{NLS}}
	- \frac i 2 \bigg(y \cdot \nabla Q_{\mathrm{NLS}}
	+ \frac 2{p-1} Q_{\mathrm{NLS}} \bigg)
	- \frac{\omega_0}{2}\nabla Q_{\mathrm{NLS}}
	+ |Q_{\mathrm{NLS}}|^{4/n} Q_{\mathrm{NLS}}
	=0.
	\label{eq:5}
	\end{align}
Especially, they showed that
there exists some radial function $\widetilde Q_{\mathrm{NLS}}$
satisfying \eqref{eq:5} with $Q_{\mathrm{NLS}}(y) = \widetilde Q_{\mathrm{NLS}}(|y|)$
with asymptotics
	\begin{align}
	\widetilde Q_{\mathrm{NLS}}(r)
	= C \frac{1}{r^{n/2}} \sin \bigg( \frac{r^2}{2}
	- \omega_0 \log \bigg(\frac{r}{2\sqrt{2}} \bigg) + \theta_0 \bigg)
	\exp \bigg( i \frac{r^2}{2} \bigg) + O(r^{-2-n/2})
	\label{eq:6}
	\end{align}
with some constant $\theta_0$ as $r \to \infty$.
We remark that the logarithmic phase correction of \eqref{eq:6}
comes from linear part of the ODE \eqref{eq:5}.
In order to show asymptotics \eqref{eq:6},
they deformed $Q_{\mathrm{NLS}}(y)$
to $e^{i y^2/8} A_{\mathrm{NLS}}(y)$
with $y = t^{-1/2} x$.
By this deformation,
\eqref{eq:6} implies that $A_{\mathrm{NLS}}$ satisfies
	\begin{align}
	\Delta A_{\mathrm{NLS}}(y)
	+ \frac{y^2}{16} A_{\mathrm{NLS}}(y)
	- \frac{\omega_0}{2} A_{\mathrm{NLS}}(y)
	+ |A_{\mathrm{NLS}}(y)|^{4/n} A_{\mathrm{NLS}}(y)
	=0.
	\label{eq:7}
	\end{align}
Moreover,
they further rewrite $A_{\mathrm{NLS}}(y)$ as $|y|^{-n/2} B_{\mathrm{NLS}}(|y|^2/8)$,
where $B_{\mathrm{NLS}}$ satisfies
	\begin{align}
	&B_{\mathrm{NLS}}''(\eta) + B_{\mathrm{NLS}}(\eta)
	- \frac{\omega_0}{\eta} B_{\mathrm{NLS}}(\eta)
	\nonumber\\
	&=
	\frac{n(n-4)}{16 \eta^2} B_{\mathrm{NLS}}(\eta)
	- \frac{1}{4\eta^2} |B_{\mathrm{NLS}}(\eta)|^{4/n} B_{\mathrm{NLS}}(\eta).
	\label{eq:8}
	\end{align}
Then it is shown that solutions to \eqref{eq:8} behave as
$C \sin(\eta - \frac{\omega_0}{2} \log \eta + \omega_1)$ as $\eta \to \infty$
with some constants $C$ and $\omega_1$, which implies the asymptotics of \eqref{eq:6}.
The motivation to deform \eqref{eq:7} to \eqref{eq:8}
may come from the form of the solutions to linearized equation of \eqref{eq:7}
with $\omega_0=0$.
In one dimensional case,
solutions to linearized equation of \eqref{eq:7} are given as follows:
%]]]

%[[[ Proposition : Linea Solutions
\begin{Proposition}[{\cite[8.491 7]{T}}]
\label{Proposition:1.1}
Let $g$ be a real analytic function satisfying
	\begin{align}
	\begin{cases}
	g''(y) + \frac{y^2}{16} g(y) = 0,
	&\qquad y \in \mathbb R,\\
	g(0) = g_0,
	\quad
	g'(0) = g_1.
	\end{cases}
	\label{eq:9}
	\end{align}
Then $g$ is given by
	\begin{align}
	g(y)
	= g_0 G_{\mathrm{even}}(y) + g_1 G_{\mathrm{odd}}(y),
	\label{eq:10}
	\end{align}
where
	\begin{align}
	G_{\mathrm{even}}(y)
	= \frac{1}{2} \Gamma(\frac 3 4) |y|^{1/2} J_{-1/4}(\frac{y^{2}}{8}),
	\qquad
	G_{\mathrm{odd}}(y)
	= 2 \Gamma(\frac{5}{4}) y|y|^{-1/2} J_{1/4}(\frac{y^{2}}{8}).
	\label{eq:11}
	\end{align}
Here $\Gamma$ is the gamma function and
$J_{\nu}$ is the Bessel function of the first kind of order $\nu$.
Moreover, for $|y| \gg 1$,
	\begin{align}
	G_{\mathrm{even}}(y)
	&= 2 \Gamma(\frac 3 4) \pi^{-1/2}
	|y|^{-1/2}\cos\bigg(\frac{y^{2}}{8}-\frac{\pi}{8}\bigg)
	+ O(|y|^{-3/2}),
	\label{eq:12}\\
	G_{\mathrm{odd}}(y)
	&= 8 \Gamma (\frac 5 4) \pi^{-1/2}
	y|y|^{-3/2}\cos\bigg(\frac{y^{2}}{8}-\frac{3\pi}{8}\bigg)
	+ O(|y|^{-3/2}).
	\label{eq:13}
	\end{align}
\end{Proposition}
%]]]

%[[[ Here, we recall that
\noindent
We recall that with real parameter $\nu$,
$J_{\nu}$ is defined as a solution to the Bessel equation
	\[
	y^{2} J_\nu'' + y J_\nu' + (y^{2}-\nu^{2})J_\nu=0.
	\]
If one regards the nonlinearity of \eqref{eq:7} as a perturbation,
the decay rate of solutions may be kept but the phase of asymptotics may be modified.
Therefore,
from \eqref{eq:12} and \eqref{eq:13},
the deformation from $A_{\mathrm{NLS}}$ to $B_{\mathrm{NLS}}$ arises naturally.
%]]]

%[[[  In this paper,
In this paper,
we study the profile function $Q$ of \eqref{eq:3} in a similar manner.
Since the Cauchy problems \eqref{eq:1} and \eqref{eq:4} share the linear part,
we also expect that $Q$ is asymptotically close to \eqref{eq:10}.
Therefore,
the profile function $Q$ will be looked for in a function space $X$ defined by
	\[
	X =
	\{ f \in C^2(\mathbb R)
	\ ; \ \langle \cdot \rangle^{1/2} f,\ \langle \cdot \rangle^{-1/2} f'
	\in L^\infty(\mathbb R) \},
	\]
where
$\langle x \rangle = ( 1+ |x|^2)^{1/2}$.
We remark that $G_{\mathrm{even}}$ and $G_{\mathrm{odd}}$ belong to $X$.
%]]]

%[[[ In this paper we shall look for self-similar solutions belonging to $X$,
Our main result is the following.
%]]]

%[[[ Theorem
\begin{Theorem}
\label{Theorem:1.2}
There is $Q \in X$ such that $u$ given by \eqref{eq:3}
satisfies \eqref{eq:1} with $u_0 = Q.$
In particular,
there are real constants $Q_\pm$
and real valued functions $\omega_\pm$ and $\phi_\pm$ such that
$\lim_{\eta \to \infty} \omega_\pm(\eta) = - \infty$ and
	\[
	Q(y) =
	\begin{cases}
	Q_+ |y|^{-1/2} \cos(\omega_+(y^2)) e^{i \phi_+(y^2)} + O(|y|^{-1}),
	&\mathrm{if} \quad y \gg 1,\\
	Q_- |y|^{-1/2} \cos(\omega_-(y^2)) e^{i \phi_-(y^2)} + O(|y|^{-1}),
	&\mathrm{if} \quad y \ll -1.
	\end{cases}
	\]
Therefore, with some positive constant $C$, we have
	\begin{align}
	\begin{cases}
	|u(t,x) - Q_+ |x|^{-1/2} \cos(\omega_+(x^2/t)) e^{i \phi_+(x^2/t)}|
	\leq C t^{1/2} |x|^{-1},
	&\mathrm{if} \quad x \gg t^{1/2},\\
	|u(t,x) - Q_- |x|^{-1/2} \cos(\omega_-(x^2/t)) e^{i \phi_-(x^2/t)}|
	\leq C t^{1/2} |x|^{-1},
	&\mathrm{if} \quad x \ll - t^{1/2}.
	\end{cases}
	\label{eq:14}
	\end{align}
\end{Theorem}
%]]]

%[[[ Remark : Theomrem
\begin{Remark}
The asymptotics \eqref{eq:14} imply that
classical solutions to \eqref{eq:1} taking the form of \eqref{eq:3}
are not square-integrable like $Q_{NLS}$.
Moreover, \eqref{eq:14} also implies that
$\lim_{t \to +0} t^{-1/4} Q(t^{-1/2}x)$ does not exist for any $x$.
Therefore, we cannot extend our solutions to $t=0$.
\end{Remark}
%]]]

%[[[ In order to seek such $f$,
In order to show Theorem \ref{Theorem:1.2},
we rewrite $Q$ by $A e^{i \phi}$ with real valued functions $A$ and $\phi$.
We remark that the amplitude function $A$ possibly takes the negative value at some points.
Then the following ODEs hold
for these amplitude and phase functions:
%]]]

%[[[ Proposition : ODE
\begin{Proposition}
\label{Proposition:1.3}
Let $A$ and $\phi$ be real valued $C^2(\mathbb R)$ functions
satisfying $\phi'(0) = - 3 A(0)^2/4$.
Let $u$ be given by \eqref{eq:3} with $Q = A e^{i\phi}$
satisfying the Cauchy problem \eqref{eq:1}
with $u_0 = Q$.
Then, $Q$, $\phi$, and $A$ satisfy
	\begin{align}
	Q''(y) - \frac i 4 Q(y) - \frac i 2 y Q'(y)
	&+
	i \frac{d}{dy} (|Q(y)|^2 Q(y)) = 0,
	& y \in \mathbb R,
	\label{eq:15}\\
	\phi'(y) &= \frac y 4 - \frac 3 4 A(y)^2,
	& y \in \mathbb R,
	\label{eq:16}\\
	A''(y) + \frac{y^2}{16} A(y)
	&= \frac y 4 A(y)^3 - \frac{3}{16} A(y)^5,
	& y \in \mathbb R.
	\label{eq:17}
	\end{align}
\end{Proposition}
%]]]

%[[[ Proposition \ref{Proposition:1.3} implies that
Proposition \ref{Proposition:1.3} implies that
Theorem \ref{Theorem:1.2} is now reduced
to the global dynamics of solutions to \eqref{eq:17}.
Indeed, by \eqref{eq:16},
$Q$ takes the form
	\[
	Q(y)
	= A(y)
	\exp \bigg(i \phi(0) + i \frac{y^2}{8} - i \frac 3 4 \int_0^y A(y')^2 dy' \bigg).
	\]
Here, $X$ is again a natural space to construct solutions to \eqref{eq:17}
because $X$ is a collection of functions
decaying like solutions to the linearized problem of \eqref{eq:17}.
%]]]

%[[[ The local solvability of \eqref{eq:17} is easily seen
The local solvability of \eqref{eq:17} is easily seen
because solutions to \eqref{eq:17} may satisfy the integral form
	\begin{align}
	A(y)
	=A_0 + A_1 y
	- \int_0^y (y-y')
	\bigg(
	\frac{y'^2}{16} A(y') - \frac {y'} 4 A(y')^3
	+ \frac{3}{16} A(y')^5 \bigg) \thinspace dy'.
	\label{eq:18}
	\end{align}
Indeed, local solutions to \eqref{eq:18} are constructed by
the standard contraction argument.
Moreover, the classical Cauchy - Kovalevskaya theorem for ODE implies
that the classical local solutions to \eqref{eq:17} are analytic in a neighborhood of $0.$
For details, we refer the reader to Theorem 2.2.1 in \cite{H2}.
%]]]

%[[[ Since ODEs \eqref{eq:17} and \eqref{eq:7}
Since ODEs \eqref{eq:17} and \eqref{eq:7} are similar,
one may regard nonlinear parts of \eqref{eq:17} as perturbation
and expect that solutions to \eqref{eq:17} take the form of \eqref{eq:10} asymptotically
at least for small data.
However, the cubic nonlinearity of \eqref{eq:17}
produces a strong nonlinear effect and manifests as a non-negligible perturbation.
For $y > 0$, by rewriting $A(y) = y^{-1/2} B(y^2/8)$,
\eqref{eq:17} is deformed as
	\begin{align}
	B''(\eta) + B(\eta)
	=
	\frac 1 {2\eta} B(\eta)^3
	- \frac 3 {16 \eta^{2}} B(\eta)
	- \frac{3}{64 \eta^2} B(\eta)^5.
	\label{eq:19}
	\end{align}
By the Duhamel principle,
\eqref{eq:19} is rewritten as
	\begin{align}
	B(\eta)
	&=
	B_0 \cos(\eta)
	+ B_1 \sin(\eta)
	\nonumber\\
	&+ \int_0^\eta \sin(\eta-\eta')
	\bigg(
	\frac 1 {2\eta'} B(\eta')^3
	- \frac 3 {16 \eta'^{2}} B(\eta')
	- \frac{3}{64 \eta'^2} B(\eta')^5
	\bigg) d\eta',
	\label{eq:20}
	\end{align}
where $B_0 = A(0)$ and $B_1=A'(0)$.
If the nonlinearity is negligible,
one may construct $B$ by the standard iteration argument.
On the other hand,
if $A$ behaves like \eqref{eq:10} asymptotically,
$B$ behaves like $C_B \sin(\eta+\omega_B)$
with some constants $C_B$ and $\omega_B$.
Then by a direct computation, we see that
	\begin{align}
	\bigg\| \int_0^\cdot \sin(\cdot-\eta') \sin(\eta'+\omega_B)^3 dy'
	\bigg\|_{L^\infty(\mathbb R)}
	= \infty.
	\label{eq:21}
	\end{align}
Therefore, it is impossible to construct $A$ in $X$
by the standard iteration argument starting from a function of $X$ generally.
In addition, this computation also shows that
$A$ cannot behave as solutions to \eqref{eq:9} asymptotically.
Otherwise, \eqref{eq:20} loses its sense on $\mathbb R$.
Here, we also remark that the asymptotic analysis of \cite{KW94}
relies on the integrability of the perturbation part of the ODE \eqref{eq:8}.
Therefore, \eqref{eq:21} also implies that
we need more careful asymptotic analysis for \eqref{eq:17}.
%]]]

%[[[ Moreover, solutions to \eqref{eq:17} are not symmetric.
Moreover, solutions to \eqref{eq:17} are not symmetric, namely, not radial.
In the case where $y < 0$,
we can regard \eqref{eq:17} as a relation between second derivative of solutions
and a collection of positive potentials.
So, as we see Lemma \ref{Lemma:4.1} below,
solutions have a priori bounds for $y < 0$ with arbitrary initial data.
On the other hand,
in the case where $y > 0$,
the cubic nonlinearity is regarded as a negative potential.
This means that the cubic nonlinearity may enlarge solutions arbitrarily.
Indeed, according to some numerical experiments,
solutions to \eqref{eq:17} seem to be unbounded
when initial data is not small.
Therefore we expect only small data global existence for \eqref{eq:17}.
We prove:
%]]]

%[[[ Proposition : Main ODE
\begin{Proposition}
\label{Proposition:1.4}
There is a small positive number $\varepsilon_0$
such that for any $0 < A_0 < \varepsilon_0$,
\eqref{eq:17} possesses a unique solution $A \in X$
with $A(0) = A_0$ and $A'(0) = 0$.
Particularly,
with some constants $Q_{\pm}$ and $\omega_{\pm,0}$,
	\begin{align}
	A(y)
	=
	\begin{cases}
	Q_{+} |y|^{-1/2}
	\cos (\omega_{+,0} - \frac{y^2}{8}
	+ \frac{3}{8} Q_+^2 \log(\frac{y^2}{8}) )
	+ O(y^{-1}),
	&\mathrm{if} \quad y \gg 1,\\
	Q_{-} |y|^{-1/2}
	\cos (\omega_{-,0} - \frac{y^2}{8}
	- \frac{3}{8} Q_-^2 \log (\frac{y^2}{8} ) )
	+ O(y^{-1}),
	&\mathrm{if} \quad y \ll -1.
	\end{cases}
	\label{eq:22}
	\end{align}
\end{Proposition}
%]]]

%[[[ In order to prove \eqref{eq:22},
In order to prove \eqref{eq:22},
we look for real valued functions $R_\pm$ and $\omega_\pm$
satisfying
	\[
	\begin{cases}
	(R_+(\eta) \cos(\omega_+(\eta)), R_+(\eta) \sin(\omega_+(\eta)))
	= (|y|^{1/2}A(y),(|y|^{1/2}A(y))^\prime),
	& \mathrm{for} \quad y > 1,\\
	(R_-(\eta) \cos(\omega_-(\eta)), R_-(\eta) \sin(\omega_-(\eta)))
	= (|y|^{1/2}A(y),(|y|^{1/2}A(y))^\prime),
	& \mathrm{for} \quad y < -1,\\
	\end{cases}
	\]
with $\eta = y^2/8$.
Then, one can show that the phase function may satisfy
	\begin{align}
	\omega_\pm'(\eta)
	&= -1 \pm \frac{Q_{\pm}^2 \cos(\omega_\pm(\eta))^4}{\eta} + O(\eta^{-3/2})
	\nonumber\\
	&= -1 \pm \frac{3 Q_{\pm}^2}{8 \eta}
	\pm \frac{ Q_{\pm}^2}{8 \eta}
	\big( 4 \cos(2\omega_\pm(\eta)) + \cos(4\omega_\pm(\eta)) \big)
	+ O(\eta^{-3/2})
	\label{eq:23}
	\end{align}
for $\eta > 1$.
Again, the second and third terms on the RHS
of the last equality of \eqref{eq:23} is not absolutely integrable
on $\mathbb R$.
In this paper, we prove that
we get logarithmic phase correction from the second term
on the RHS of the last equality of \eqref{eq:23}
and the remainder is under control.
%]]]

%[[[ Here, we also remark that
Here, we also remark that
Cazenave and Weissler \cite{CW98} studied another type of self-similar solutions
to the Cauchy problem \eqref{eq:4}.
In particular, they consider the solution propagating from the initial data
$u_0(x) = |x|^{-n/2}$.
Since \eqref{eq:4} and this initial data are invariant
under the scaling transformations
$u \mapsto T_{n,\lambda}(u)$ and
$u_0 \mapsto T_{n,\lambda}^0(u_0)$ defined by
	\[
	T_{n,\lambda} (u)(t,x)
	= \lambda^{n/2} u(\lambda^2 t, \lambda x),\quad
	T_{n,\lambda}^0 (u_0)
	= \lambda^{n/2} u_0(\lambda \cdot)
	\]
for $\lambda > 0$, respectively,
if the solution is unique,
by substituting $\lambda = t^{-1/2}$,
solutions take the form of \eqref{eq:3}.
Since the Cauchy problem \eqref{eq:1} is also invariant
under the scaling transformation $T_{1,\lambda}$,
one may expect self-similar solutions constructed similarly.
However,
this type of self-similar solutions is outside of the scope of this paper.
%]]]

%[[[ This paper is organized as follows:
This paper is organized as follows:
In Section \ref{section:2},
we show Proposition \ref{Proposition:1.3}.
In Section \ref{section:3},
we review the proof of Proposition \ref{Proposition:1.1}.
In Section \ref{section:4},
we show Proposition \ref{Proposition:1.4}.
Particularly,
we, at first, show local existence of solutions by the standard contraction argument.
Then a priori bounds of solutions to \eqref{eq:17} are obtained
by introducing some modified energies.
We remark that, since blow-up alternative holds for \eqref{eq:17},
the a priori bounds imply the global existence of solutions to \eqref{eq:17}.
At last, we show asymptotic behavior of solutions to \eqref{eq:17}.
%]]]
%]]]

%[[[ section : Some ODEs for self-similar solutions
\section{Some ODEs for Self-similar Solutions}
\label{section:2}
Here we prove Proposition \ref{Proposition:1.3}.
By substituting \eqref{eq:3},
\eqref{eq:1} is rewritten as
	\begin{align*}
	& - \frac i 4 t^{-5/4} Q(t^{-1/2} x)
	- \frac i 2 t^{-5/4} t^{-1/2} x Q'(t^{-1/2} x)
	+ t^{-5/4} Q''(t^{-1/2} x)\\
	&= - i t^{-3/4} \partial_x (|Q(t^{-1/2}x)|^{2} Q(t^{-1/2}x) )\\
	&= - i t^{-5/4}(|Q|^2 Q)'(t^{-1/2}x).
	\end{align*}
By multiplying $t^{5/4}$ on both hand sides and rewriting $t^{-1/2}x$ as $y$,
we obtain \eqref{eq:15}.
Substituting $Q = A e^{i \phi}$ in \eqref{eq:15},
multiplying $e^{-i \phi}$ on both hand sides of \eqref{eq:15},
taking the imaginary part of the resulting equation,
and multiplying $A$,
we obtain
	\begin{align*}
	& 2 A(y) A'(y) \phi'(y) + A(y)^2 \phi''(y)
	- \frac 1 4 A(y)^2 - \frac y 2 A(y)A'(y)\\
	&=
	\frac{d}{dy} ( A(y)^2 \phi'(y) )
	- \frac 1 4 \frac{d}{dy} ( y A(y)^2)\\
	&= - \frac 3 4 \frac{d}{dy} A(y)^4.
	\end{align*}
Therefore,
	\[
	A(y)^2 \phi'(y)
	- \frac y 4 A(y)^2
	= - \frac 3 4 A(y)^4,
	\]
which together with the assumption $\phi'(0) = 3 A(0)^2/4$ implies \eqref{eq:16}.
Moreover,
multiplying $e^{-i \phi}$ on both hand sides of \eqref{eq:15} and
taking the real part of the resulting equation,
we have
	\begin{align}
	A''(y) - A(y) \phi'(y)^2 + \frac y 2 A(y) \phi'(y)
	= A(y)^3 \phi'(y).
	\label{eq:24}
	\end{align}
By \eqref{eq:16}, the LHS of \eqref{eq:24} is rewritten as
	\begin{align*}
	&A''(y) + \frac y 2  A(y) \phi'(y) - A(y) \phi'(y)^2\\
	&= A''(y) + \frac{y^2}{8} A(y) - \frac{3y}{8} A(y)^3
	- \frac{y^2}{16} A(y) + \frac{3y}{8} A(y)^3 - \frac{9}{16} A(y)^5,
	\end{align*}
while the RHS is rewritten as
	\[
	\frac y 4 A(y)^3 - \frac 3 4 A(y)^5.
	\]
This implies \eqref{eq:17}.
%]]]

%[[[ section : Solutions for linearlized problem
\section{Solutions for Linearized Problem}
\label{section:3}
%[[[ \begin{proof}[Proof of Proposition \ref{Proposition:1.1}]
\begin{proof}[Proof of Proposition \ref{Proposition:1.1}]
Proposition \ref{Proposition:1.1} is a conclusion of \cite[8.491 7]{T}.
But for reader's convenience, we show the proof.
In order to show \eqref{eq:11},
we see only the case where $y>0$,
because \eqref{eq:9} is symmetric
under reflection and one can consider even and odd extensions.
Then solutions to \eqref{eq:9} for $y > 0$ are given
by a linear combination of
	\[
	y^{1/2} J_{\pm 1/4} ( \frac{y^2}{8}).
	\]
For details, see \cite[8.491 7]{T}.
$J_{\nu}$ has the following expansion:
	\[
	J_{\nu}(z)
	=\bigg(\frac{z}{2}\bigg)^{\nu}
	\sum_{k=0}^{\infty}\frac{(-1)^{k}}{k!\Gamma(k+\nu+1)}\bigg(\frac{z}{2}\bigg)^{k},
	\qquad \mathrm{for} \quad |\mathrm{Arg} z| < \pi.
	\]
For details, see \cite[8.440]{T}.
Therefore, $G_{\mathrm{even}}$ and $G_{\mathrm{odd}}$ arise naturally
as even and odd solutions to \eqref{eq:9}, respectively.
Moreover, they are analytic and have expansions
	\[
	G_{\mathrm{even}}(y)
	=\sum_{k=0}^{\infty}\frac{(-1)^{k} 2^{-4k} \Gamma(3/4)}{k!\Gamma(k+3/4)} y^{2k},
	\quad
	G_{\mathrm{odd}}(y)
	=\sum_{k=0}^{\infty}\frac{(-1)^{k}2^{-4k} \Gamma(5/4)}{k!\Gamma(k+5/4)} y^{2k+1}.
	\]
$J_{\nu}$ has an asymptotic behavior
	\begin{align}
	J_{\nu}(\eta)
	= \sqrt{\frac{2}{\pi}} \eta^{-1/2}\cos\bigg(\eta-\frac{\pi}{2}\nu-\frac{\pi}{4}\bigg)
	+ O(\eta^{-1})
	\label{eq:25}
	\end{align}
as $\eta \to \infty$.
See \cite[8.451 1]{T}.
Therefore, \eqref{eq:12} and \eqref{eq:13} hold.
%]]]

%[[[ At last,
At last,
it is also known that
	\[
	J_{\nu}'(\eta)
	= \frac{J_{\nu-1}(\eta) - J_{\nu+1}(\eta)}{2}.
	\]
For details, see \cite[8.471 2]{T}.
Combining this and \eqref{eq:25}, the bound
	\[
	\bigg| J_{\nu}'(\frac{y^2}{8}) \bigg| \leq C
	\]
holds when $|y| \gg 1$.
Therefore, this and \eqref{eq:25} imply that
$G_{\mathrm{even}}$ and $G_{\mathrm{odd}}$ belong to $X$.
%]]]
\end{proof}
%]]]

%[[[ \section{Study of Solutions to \eqref{eq:17}}
\section{Study of Solutions to \eqref{eq:17}}
\label{section:4}
%[[[ In this section,
In this section,
we prove Proposition \ref{Proposition:1.4}.
We remark that
a standard contraction argument implies that
for any $(A_0,A_1)$,
there exists $\delta >0$ such that
\eqref{eq:20} posses a unique $C^2$ solution on $(-\delta,\delta)$.
%]]]

%[[[ \subsection{A Priori Control of Solutions to \eqref{eq:17}}
\subsection{A Priori Control of Solutions to \eqref{eq:17}}
%[[[ In this subsection, we show solutions to \eqref{eq:17} belong to $X$,
In this subsection,
we prove that solutions to \eqref{eq:17} with small initial data belong to $X$.
This a priori control and the standard contraction argument
imply global existence of solutions to \eqref{eq:17} for small initial data.
We divide the proof into two parts:
the cases where $y<0$ and $y>0$.
%]]]

%[[[ \subsubsection{The Case $y < 0$}
\subsubsection{The Case $y < 0$}
%[[[ We, at first,
For simplicity, we put $y$ as $-y$,
namely, consider
	\begin{align}
	\begin{cases}
	\widetilde A''(y)
	= - \Big( \frac{y^2}{16} \widetilde A(y)
	+ \frac{y}{4} \widetilde A^3(y) + \frac{3}{16} \widetilde A^5(y) \Big),
	&\mathrm{if} \quad y > 0,\\
	\widetilde A(0)=A_0,\quad \widetilde A'(0)=A_1,
	\end{cases}
	\label{eq:26}
	\end{align}
where $\widetilde A = A(-\cdot)$.
We, at first,
show a uniform bound of solutions to
second order ODEs with positive potentials.
%]]]

%[[[ Lemma :
\begin{Lemma}
\label{Lemma:4.1}
Let $K \geq 0$ and $(V_k)_{k=0}^K$
be a sequence of nonnegative increasing functions on $\lbrack 0, \infty)$.
We assume that there exists $x_0 > 0$ such that
	\[
	V_0(x_0)= \nu > 0.
	\]
Let $V(x,y) = \sum_{k=0}^K V_k(x) y^k$
and let $f \in C^2$ satisfy
	\begin{align}
	\begin{cases}
	f''(x)
	= - V(x, |f(x)|^{2}) \cdot f(x),
	&\qquad x \geq 0,\\
	f(0) = f_0 > 0,
	\qquad f'(0) = 0.
	\end{cases}
	\label{eq:27}
	\end{align}
Then, there exist unbounded increasing sequences
$(m_j)_{j=1}^\infty \subset (0,\infty)$ and
$(n_j)_{j=1}^\infty \subset (0,\infty)$ such that
$m_j$ and $n_j$ are the maximal points of $f^2$ and $(f')^2$,
respectively.
We also put $m_0 = 0$.
Moreover, for any $j \geq 1$,
$m_{j-1} < n_j < m_j$ and
	\begin{align}
	\sum_{k=0}^K \frac{V_k(n_j)}{k+1} |f(m_j)|^{2(k+1)}
	&\leq f'(n_j)^2
	\leq \sum_{k=0}^K \frac{V_k(n_j)}{k+1} |f(m_{j-1})|^{2(k+1)},
	\label{eq:28}\\
	\sum_{k=0}^K \frac{V_k(m_j)}{k+1} |f(m_j)|^{2(k+1)}
	&\geq f'(n_j)^2
	\geq \sum_{k=0}^K \frac{V_k(m_{j-1})}{k+1} |f(m_{j-1})|^{2(k+1)}.
	\label{eq:29}
	\end{align}
\end{Lemma}

%[[[ Remark :
\begin{Remark}
\label{Remark:4.1}
If $(V_k)_{k=0}^K$ is a sequence of constants,
\eqref{eq:28} and \eqref{eq:29} are equivalent to
the classical energy conservation
	\[
	\sum_{k=0}^K \frac{V_k(0)}{k+1} |f(m_j)|^{2(k+1)}
	= f'(n_{j'})^2
	\]
for any $j \geq 0$ and $j' \geq 1$.
Moreover,
\eqref{eq:28} implies that
$|f(m_j)|$ decreases as $j \to \infty$,
because $y \mapsto \sum_{k=0}^K \frac{V_k(n_j)}{k+1} y^{2(k+1)}$ is increasing
at least when $n_j \geq x_0$.
Therefore, $|f(x)| < f(x_0)$ for any $x > x_0$.
On the other hand,
\eqref{eq:29} implies that $f'(n_j)^2$ increases as $j \to \infty$.
\end{Remark}
%]]]

\begin{proof}
%[[[ At fist, we show the existence of $(m_j)_{j=0}^\infty$ and $(n_j)_{j=1}^\infty$,
At fist, we show the existence of $(m_j)_{j=0}^\infty$ and $(n_j)_{j=1}^\infty$,
i.e. $f$ keeps oscillating.
Let
	\[
	I_0 = \{ x > m_0 = 0 \ ; \ f(\xi) > 0, \ \forall \xi \in (m_0,x) \}.
	\]
Then $I_0 \neq \emptyset$ because of the continuity of $f$.
Moreover $I_0$ is bounded.
Otherwise,
for $x \geq 2 x_0$,
	\begin{align*}
	f'(x)
	&= f'(x_0) - \int_{x_0}^{x} V(x,|f(x)|^2) \cdot f(x) dx\\
	&\leq - \nu \int_{x_0}^{2 x_0} f(x) dx\\
	&\leq - \nu x_0 f(2 x_0) < 0,
	\end{align*}
which implies
	\[
	f(x) = f(x_0) + \int_{x_0}^x f'(y) dy \leq f(x_0) - \nu x_0 f(2 x_0) (x-x_0)
	\to - \infty \quad \mbox{as} \quad x \to \infty.
	\]
This contradicts that $f$ is positive on $I_0$.
Then let $n_1 = \sup I_0$.
Then $f'(n_1) < 0$ because $f'' < 0$ on the interior of $I_0$.
Moreover, since the sign of $f''$ and that of $f$ are the opposite of each other,
$n_1$ is a maximal point of $f'^2$.
%]]]

%[[[ Next, let
Next, let
	\[
	J_1 = \{ x > n_1 \ ; \ f'(\xi) < 0, \ \forall \xi \in (n_1,x) \}.
	\]
By the continuity of $f'$ and $f''$,
$J_1 \neq \emptyset$.
Again $J_1$ is shown to be bounded.
Otherwise, $f(x) < 0$ and $f'(x) <0$ for $x > x_1 = 2 \max(n_1,x_0)$.
Then
	\[
	f''(x) \geq - \nu f(x_1) > 0,
	\]
which implies
	\[
	f'(x) > f'(x_1) + \nu |f(x_1)|(x-x_1)
	\to \infty \quad \mbox{as} \quad x \to \infty.
	\]
This contradicts that $\sup J_1 = \infty$.
Let $m_1 = \sup J_1$.
Then the continuity of $f'$ and \eqref{eq:27} imply
that $f'(m_1)=0$ and $f''(m_1) > 0$, respectively.
Therefore $m_1$ is a maximal point of $f^2$.
%]]]

%[[[ Next, let
Next, let
	\[
	I_1 = \{ x > m_1 \ ; f(\xi) < 0, \ \forall \xi \in (m_1,x) \}.
	\]
Then $\widetilde f = -f(\cdot+m_1)$ satisfies that
$\widetilde f(0) = - f(m_1) > 0$ and $\widetilde f'(0) = - f'(m_1) = 0$.
Moreover, $\widetilde f$ satisfies \eqref{eq:27} with
$\widetilde V(x,y) = V(x+m_1,y)$,
where $\widetilde V_k = V_k(\cdot+m_1)$ is nonnegative increasing
for any $k \geq 0$ and $\widetilde V_{0} (x_0) \geq V_{0} (x_0) > 0$.
Therefore, the boundedness $I_1$ is shown by the same argument as that for $I_0$.
Then let $n_2 = \sup I_1$.
Similarly, let
	\[
	J_2 = \{ x > n_2 \ ; \ f'(\xi) > 0, \ \forall \xi \in (n_2,x) \}.
	\]
Then, $J_2$ is shown to be bounded in the same manner as $J_1$.
%]]]

%[[[ By repeating this argument,
By repeating this argument,
we can define a bounded interval $I_j$ and $J_j$ by
	\begin{align*}
	I_j &= \{ x > m_j \ ; \ (-1)^j f(\xi) > 0, \ \forall \xi \in (m_{j},x) \},\\
	J_j &= \{ x > n_j \ ; \ (-1)^{j} f'(\xi) > 0, \ \forall \xi \in (n_j,x) \},
	\end{align*}
where $n_{j+1} = \sup I_{j}$ and $m_j = \sup J_j$ for $j \geq 2$,
recursively and respectively.
Obviously $m_j < n_{j+1} < m_{j+1}$ for any $j \geq 0$.
Therefore, if $(m_j)_{j=0}^\infty$ converges,
then so does $(n_j)_{j=1}^\infty$.
In this case, let $x_2 = \lim_{j \to \infty} m_j$.
Then the continuity of $f$ implies that $f(x_2) = f'(x_2) = 0$.
This and the contraction argument imply that $f \equiv 0$.
This is a contradiction
and therefore $(m_j)_{j=0}^\infty$ and $(n_j)_{j=1}^\infty$ are unbounded.
%]]]

%[[[ Now, we prove \eqref{eq:28}.
Now, we prove \eqref{eq:28}.
By multiplying $f'$ on both sides of \eqref{eq:27},
	\[
	\frac 1 2 \frac{d}{dx} (f'^2)(x)
	= - V(x,f(x)^2) \frac 1 2 \frac{d}{dx} f(x)^2
	= - \frac 1 2 \sum_{k=0}^K \frac{V_k(x)}{k+1} \frac{d}{dx} f(x)^{2(k+1)}.
	\]
We fix $j \geq 1$.
Since $f^2$ is decreasing on $I_j = (m_{j},n_{j+1})$,
	\begin{align*}
	\frac 1 2 f'(n_j)^2
	&= \frac 1 2 f'(n_j)^2 - \frac 1 2 f'(m_{j-1})^2\\
	&= - \frac 1 2
	\sum_{k=0}^K \int_{m_{j-1}}^{n_j} \frac{V_k(x)}{k+1} \frac{d}{dx} f(x)^{2(k+1)} dx\\
	&\leq \frac 1 2
	\sum_{k=0}^K \frac{V_k(n_j)}{k+1}
	\bigg| \int_{m_{j-1}}^{n_j} \frac{d}{dx} f(x)^{2(k+1)} dx \bigg|\\
	&\leq \frac 1 2
	\sum_{k=0}^K \frac{V_k(n_j)}{k+1} f(m_{j-1})^{2(k+1)}.
	\end{align*}
Moreover, since $|f|$ is increasing on $J_j = (n_j,m_j)$,
	\begin{align*}
	- \frac 1 2 f'(n_j)^2
	&= \frac 1 2 f'(m_j)^2 - \frac 1 2 f'(n_{j})^2\\
	&= - \frac 1 2
	\sum_{k=0}^K \int_{n_{j}}^{m_j} \frac{V_k(x)}{k+1} \frac{d}{dx} f(x)^{2(k+1)} dx\\
	&\leq - \frac 1 2
	\sum_{k=0}^K \frac{V_k(n_j)}{k+1}
	\bigg| \int_{n_{j}}^{m_j} \frac{d}{dx} f(x)^{2(k+1)} dx \bigg|\\
	&\leq - \frac 1 2
	\sum_{k=0}^K \frac{V_k(n_j)}{k+1} f(m_{j})^{2(k+1)}.
	\end{align*}
Combining these two estimates, we obtain \eqref{eq:28}.
The inequalities \eqref{eq:29} are shown similarly so we omit the proof.
\end{proof}
%]]]

%]]]

%[[[ Proposition : Negative Control
\begin{Proposition}
\label{Proposition:4.2}
Let $\widetilde A$ be a solution to \eqref{eq:26}.
Then, for any $y > n_1 > 0$,
	\begin{align*}
	&\frac{1}{2y} \widetilde A'(y)^2 + \frac{y}{32} \widetilde A(y)^2\\
	&\leq \frac{n_1}{32} A_0^2 + \frac{1}{16} A_0^4 + \frac{1}{32 n_1} A_0^6
	+ \frac{1}{2 n_1} A_0
	\bigg( \frac{1}{16} A_0^2 + \frac{1}{8 n_1} A_0^4  + \frac{1}{16 n_1^{2}} A_0^6
	\bigg)^{1/2},
	\end{align*}
where $n_{1}$ is the minimal positive zero of $\widetilde A$.
\end{Proposition}
\begin{proof}
%[[[ Step 1
\noindent{\bf Step 1:} $\widetilde A'(y) \lesssim 1$ for small $y$.\\
By Remark \ref{Remark:4.1},
$\widetilde A(y) \leq A_0$ for $y \geq 0$
and
	\[
	\widetilde A'(y)^2
	\leq \frac{n_{1}^2}{16} A_0^2
	+ \frac{n_{1}}{8} A_0^4
	+ \frac{1}{16} A_0^6,
	\]
for $0 < y \leq n_1$.
%]]]
\bigskip

%[[[ Step 2
\noindent{\bf Step 2:} $\widetilde A'(y) \lesssim y$ for any $y \geq n_1$.\\
Let $E_1$ be a modified energy defined by
	\[
	E_1(y)
	= \frac{1}{2} \widetilde A'(y)^2
	+ \frac{y^2}{32} \widetilde A(y)^2
	+ \frac{y}{16} \widetilde A(y)^4
	+ \frac{1}{32} \widetilde A(y)^6.
	\]
Then
	\[
	\frac{d}{dy} \bigg( \frac{1}{y^{2}} E_1(y) \bigg)
	= - \frac{1}{y^{3}} \widetilde A'(y)^2
	- \frac{1}{16 y^{2}} \widetilde A(y)^4
	- \frac{1}{16 y^{3}} \widetilde A(y)^6
	\leq 0.
	\]
Therefore, for $y \geq n_1$,
	\begin{align*}
	\frac{1}{2 y^{2}} \widetilde A'(y)^2
	&\leq \frac{1}{y^2} E_1(y)
	\leq \frac{1}{n_1^2} E_1(n_1)\\
	&\leq
	\frac{1}{2 n_1^{2}} \widetilde A'(n_1)^2
	\leq \frac{1}{32} A_0^2
	+ \frac{1}{16 n_{1}} A_0^4
	+ \frac{1}{32  n_1^{2}} A_0^6.
	\end{align*}
%]]]
\bigskip

%[[[ Step 3
\noindent{\bf Step 3:} $\widetilde A(y) \lesssim y^{-1/2}$ and $\widetilde A'(y) \lesssim y^{1/2}$.\\
Let $E_2$ be another modified energy defined by
	\[
	E_2(y)
	= \frac 1 y E_1(y)
	+ \frac 1 {2 y^{2}} \widetilde A(y) \widetilde A'(y) + \frac 1 {2 y^{3}} \widetilde A(y)^2.
	\]
Then, for $y > n_1$,
	\begin{align*}
	E_2'(y)
	&= -\frac{1}{2 y^{2}} \widetilde A'(y)^2
	+ \frac{1}{32} \widetilde A(y)^2 - \frac{1}{32y^{2}} \widetilde A(y)^6\\
	&- \frac{1}{y^{3}} \widetilde A(y) \widetilde A'(y)
	+ \frac 1 {2 y^{2}} \widetilde A'(y)^2
	+ \frac 1 {2 y^{2}} \widetilde A(y) \widetilde A''(y)\\
	&- \frac 3 {2 y^{4}} \widetilde A(y)^2 + \frac{1}{y^{3}} \widetilde A(y) \widetilde A'(y)\\
	&= \frac{1}{32} \widetilde A(y)^2 - \frac{1}{32 y^{2}} \widetilde A(y)^6\\
	&- \frac{1}{32} \widetilde A(y)^2 - \frac{1}{8 y} \widetilde A(y)^4 - \frac{3}{32 y^{2}} \widetilde A(y)^6
	- \frac 3 {2 y^{4}} \widetilde A(y)^2\\
	&=
	- \frac{1}{8 y^{2}} \widetilde A(y)^6
	- \frac{1}{8y} \widetilde A(y)^4
	- \frac 3 {2 y^{4}} \widetilde A(y)^2
	\leq 0.
	\end{align*}
Therefore, for $y > n_1$,
	\begin{align*}
	&\frac{1}{2y} \widetilde A'(y)^2
	+ \frac{y}{32} \widetilde A(y)^2 + \frac{1}{16} \widetilde A(y)^4
	+ \frac{1}{32 y} \widetilde A(y)^6
	+ \frac 1 {2 y^{3}} \widetilde A(z)^2\\
	&\leq E_2(n_1) + \frac 1 {2 y^{2}} |\widetilde A(y) \widetilde A'(y)|\\
	&\leq \frac 1{n_1} E_1(n_1) + \frac 1 {2 y^{2}} |\widetilde A(y) \widetilde A'(y)|\\
	&\leq \frac{1}{2 n_1} \widetilde A'(n_1)^2
	+ \frac{1}{2 y} A_0
	\bigg( \frac{1}{16} A_0^2 + \frac{1}{8n_1} A_0^4 + \frac{1}{16 n_1} A_0^6
	\bigg)^{1/2}\\
	&\leq \frac{n_1}{32} A_0^2 + \frac{1}{16} A_0^4 + \frac{1}{32 n_1} A_0^6
	+ \frac{1}{2 n_1} A_0
	\bigg( \frac{1}{16} A_0^2 + \frac{1}{8 n_1} A_0^4  + \frac{1}{16 n_1} A_0^6
	\bigg)^{1/2}.
	\end{align*}
\end{proof}
%]]]
%]]]
%]]]

%[[[ \subsubsection{The Case $y > 0$}
\subsubsection{The Case $y > 0$}
%[[[ In the case where $y$ is positive,
Since the RHS of \eqref{eq:17}
is a mixture of positive and negative potentials,
solutions to \eqref{eq:17} is unbounded in general.
Therefore, we consider only small data for this case.
% ]]]

%[[[ \begin{Proposition}
\begin{Proposition}
\label{Proposition:4.3}
There exists a constant $C_3 >0$ such that
for sufficiently small $\varepsilon >0$ and any $y \geq 0$,
a solution $A$ to \eqref{eq:17} with $(A_0,A_1) = (\varepsilon,0)$
satisfies
	\[
	|A(y)| \leq C_3 \varepsilon ( 1 + |y|)^{-1/2},
	\quad
	|A'(y)| \leq C_3 \varepsilon ( 1 + |y|)^{1/2}.
	\]
\end{Proposition}
%]]]

%[[[ \begin{proof}
\begin{proof}
%[[[ Step 1
\noindent{\bf Step 1:} $|A|, |A'| \lesssim \varepsilon$ for small $y$.\\
At first, we show that
there exists $y_0 \geq 1/3$ such that for any $\varepsilon < 1$,
	\[
	\sup_{y \in (0,y_0)} |A(y)|
	\leq 2 \varepsilon.
	\]
Indeed,
for $y \in (0,y_0)$,
\eqref{eq:17} implies
	\[
	|A''(y)|
	\leq \frac{1}{8} y^2 \varepsilon
	+ 2 y \varepsilon^3
	+ 6 \varepsilon^5.
	\]
By integrating this and resulting estimate, we get
	\begin{align*}
	|A'(y)|
	&\leq \frac{1}{24} y^3 \varepsilon
	+ y^2 \varepsilon^3
	+ 6 y \varepsilon^5,\\
	|A(y)|
	&\leq \varepsilon
	+ \frac{1}{96} y^4 \varepsilon
	+ \frac{1}{3} y^3 \varepsilon^3
	+ 3 y^2 \varepsilon^5.
	\end{align*}
Therefore,
if $y \leq 1/3$,
then
	\begin{align}
	|A(y)|
	\leq \varepsilon
	+ \frac{1}{96} \varepsilon
	+ \frac{1}{3} \varepsilon^3
	+ \frac{1}{3} \varepsilon^5
	< 2 \varepsilon.
	\label{eq:30}
	\end{align}
By the continuity of $A$,
if $\sup_{y \geq 0} |A(y)| > 2 \varepsilon$,
we can conclude that there exists $y_0 > 0$ so that $|A(y_0)| = 2 \varepsilon$.
Therefore \eqref{eq:30} implies that $y_0 \geq 1/3$.

%]]]
\bigskip

%[[[ Step 2
\noindent{\bf Step 2:} $A'(y) \lesssim y$.\\
Let $E_3$ be the modified energy defined by
	\[
	E_3(y)
	= \frac{1}{2} A'(y)^2
	+ \frac{y^2}{32} A(y)^2
	- \frac{y}{16} A(y)^4
	+ \frac{1}{32} A(y)^6.
	\]
Then
	\begin{align}
	\frac{d}{dy} \bigg( \frac{1}{y^{2}} E_3(y) \bigg)
	= - \frac{1}{y^{3}} A'(y)^2
	+ \frac{1}{16 y^{2}} A(y)^4
	- \frac{1}{16 y^{3}} A(y)^6
	\leq \frac{1}{16 y^{2}} A(y)^4.
	\label{eq:31}
	\end{align}
Therefore, for $y \geq y_0$,
	\[
	\frac{1}{32} A(y)^2
	\leq \frac{1}{y_0^{2}} E_3(y_0) + \frac{1}{16 y} A(y)^4
	+ \int_{y_0}^y \frac{1}{16 y^{\prime 2}} A(y^\prime)^4 dy^\prime.
	\]
By putting $F(y) = \sup_{y^\prime \in \lbrack y_0,y)} A(y^\prime)^2$,
we have $F(y_0)= 4 \varepsilon^2$ and
	\[
	F(y)
	\leq C \varepsilon^2 + \frac 2 {y_0} F(y)^2,
	\]
for sufficiently large constant $C$.
This estimate is rewritten as
	\begin{align}
	(F(y) - \frac{y_0}{4})^2 + \frac{C y_0}{2} \varepsilon^2 - \frac{y_0^2}{16} \geq 0.
	\label{eq:32}
	\end{align}
Then let $\varepsilon$ be sufficiently small so that
	\begin{align}
	4 \varepsilon^2 \leq \frac{y_0}{4},
	\label{eq:33}\\
	\frac{C y_0}{2} \varepsilon^2 \leq \frac{y_0^2}{16}.
	\label{eq:34}
	\end{align}
The inequality \eqref{eq:34} implies that \eqref{eq:32} holds only in the case where
	\[
	F(y) < \frac{y_0}{4} - \sqrt{\frac{y_0^2}{16} - \frac{C y_0}{2} \varepsilon^2},
	\qquad \mathrm{or} \qquad
	F(y) > \frac{y_0}{4} + \sqrt{\frac{y_0^2}{16} - \frac{C y_0}{2} \varepsilon^2}.
	\]
Since $F(y_0)$ satisfies \eqref{eq:32}, \eqref{eq:33} implies that for any $y \geq y_0$,
	\[
	F(y)
	< \frac{y_0}{4} - \sqrt{\frac{y_0^2}{16} - \frac{C y_0}{2} \varepsilon^2}
	< \sqrt{\frac{C y_0}{2}} \varepsilon.
	\]
Therefore $|A(y)| \leq C \varepsilon^{1/2}$.
By integrating \eqref{eq:31} with this upper bound,
we have $y^{-2} E_3(y) \leq C \varepsilon^2$
and consequently,
$|A(y)| \leq C_1 \varepsilon$ and $|A^\prime(y)| \leq C_2 \varepsilon y$ for any $y > y_0$
with some positive constants $C_1$ and $C_2$.
%]]]
\bigskip

%[[[ Step 3
\noindent{\bf Step 3:} $A(y) \lesssim y^{-1/2}$ and $A'(y) \lesssim y^{1/2}$.\\
Let $E_4$ be a modified energy defined by
	\[
	E_4(y)
	= \frac{1}{y} E_3(y)
	+ \frac 1 {2 y^{2}} A(y) A'(y) + \frac 1 {2 y^{3}} A(y)^2.
	\]
Then, for $y > y_0$,
	\begin{align*}
	E_4'(y)
	=
	- \frac{1}{8 y^{2}} A(y)^6
	+ \frac{1}{8 y} A(y)^4
	- \frac 3 {2 y^{4}} A(y)^2
	\leq \frac{1}{8 y} A(y)^4.
	\end{align*}
Therefore, for $y > y_0$, by the estimates of Step 2,
	\begin{align}
	&\frac{1}{2 y} A'(y)^2
	+ \frac{y}{32} A(y)^2 + \frac{1}{32y} A(y)^6
	+ \frac 1 {2 y^{3}} A(y)^2
	\nonumber\\
	&\leq \frac{1}{y_0} E_3(y_0) + \frac 1 {2 y^{2}} |A(y) A'(y)| + \frac{1}{16} A(y)^4
	+ \int_{y_0}^y \frac{1}{8 y^\prime} A(y^\prime)^4 dy^\prime
	\label{eq:35}\\
	&\leq C_3 \varepsilon^2 + C_4 \varepsilon^2 \log(y).
	\nonumber
	\end{align}
This implies $|A(y)| \lesssim \varepsilon y^{-1/2} ( 1 + \log y)$.
Substituting this to \eqref{eq:35},
we see that $E_4'$ is integrable on $\lbrack y_0, \infty)$,
which implies our assertion.
%]]]
\end{proof}
%]]]
%]]]

%[[[ \begin{proof}[Proof of Global existence of solutions to \eqref{eq:17}
\subsection{Proof of Global existence of solutions to \eqref{eq:17}}
Propositions \ref{Proposition:4.2} and \ref{Proposition:4.3} imply
the a priori control
	\[
	\| \langle \cdot \rangle^{1/2} A \|_{L^\infty(0,\infty)}
	+ \| \langle \cdot \rangle^{-1/2} A' \|_{L^\infty(0,\infty)}
	\leq C \varepsilon,
	\]
where $C$ is independent of $\varepsilon$.
This a priori bound implies that local solutions to \eqref{eq:17} are extended arbitrarily
and $A \in X$.
%]]]
%]]]

%[[[ subsection : Asymptotic behanior
\subsection{Asymptotic behavior}
At first, we consider the case where $y < -1$.
Again we put $-y$ as $y$ and consider a solution $\widetilde A$ to \eqref{eq:26}.
Let $B(y^2/8) = |y|^{1/2} \widetilde A(y)$.
We remark that $\widetilde A \in X$ implies that $B$ and $B'$ are bounded for $|y| > 1$.
Since
	\[
	\widetilde A''(y)
	= \frac 1 {16} y^{3/2} B''( \frac{y^2}{8})
	+ \frac 3 4 y^{-5/2} B( \frac{y^2}{8}),
	\]
the equation \eqref{eq:17} is rewritten as
	\begin{align}
	B''(\eta) + B(\eta)
	=
	- \frac 1 {2\eta} B(\eta)^3
	- \frac 3 {16 \eta^{2}} B(\eta)
	- \frac{3}{64 \eta^2} B(\eta)^5
	\label{eq:36}
	\end{align}
with $\eta = y^2/8$.
Therefore, energy of \eqref{eq:36} is computed to be
	\begin{align}
	E_B(\eta)
	&= \frac 1 2 B'(\eta)^2 + \frac 1 2 B(\eta)^2 + F_B(\eta)
	\equiv E_B,
	\label{eq:37}\\
	F_B(\eta)
	&= \frac 1 {8\eta} B^4(\eta) - \int_{\eta}^\infty
	\bigg( \frac 1 {8} B(\zeta)^4
	+ \frac 3 {16} B(\zeta) B^\prime(\zeta)
	+ \frac{3}{64} B(\zeta)^5 B^\prime(\zeta) \bigg) \frac{d \zeta}{\zeta^2}.
	\nonumber
	\end{align}
Since $B$ and $B^\prime$ are bounded, $F_B(\eta) = O(1/\eta)$.
Now, we consider real valued functions $R_B$ and $\omega$ satisfying
$B(\eta) = R_B(\eta) \cos(\omega(\eta))$ and $B'(\eta) = R_B(\eta) \sin(\omega(\eta))$.
Then \eqref{eq:37} implies
	\begin{align}
	R_B^2(\eta) = 2 E_B + O(\eta^{-1}).
	\label{eq:38}
	\end{align}
Next, we put $G(\eta) = R_B(\eta) e^{i \omega(\eta)} = B(\eta) + i B^\prime(\eta)$.
Then
	\begin{align}
	G'(\eta) + i G(\eta)
	&= R'(\eta) e^{i \omega(\eta)}
	+ i\omega'(\eta) R(\eta) e^{i \omega(\eta)} + i R(\eta) e^{i \omega(\eta)}
	\nonumber\\
	&=
	- \frac i \eta \bigg( \frac 1 {2} B(\eta)^3
	+ \frac 3 {16 \eta} B(\eta)
	+ \frac{3}{64\eta} B(\eta)^5 \bigg).
	\label{eq:39}
	\end{align}
By multiplying $e^{- i \omega(\eta)}$ both hand sides of \eqref{eq:39}
and taking real and imaginary parts of resulting identity,
	\begin{align}
	R'(\eta)
	&= - \frac{\sin(\omega(\eta))}{\eta} \bigg( \frac 1 {2} B(\eta)^3
	+ \frac 3 {16 \eta} B(\eta)
	+ \frac{3}{64\eta} B(\eta)^5 \bigg),
	\label{eq:40}\\
	\omega'(\eta)
	&= - 1 - \frac{\cos(\omega(\eta))}{\eta R(\eta)}
	\bigg( \frac 1 {2} B(\eta)^3
	+ \frac 3 {16 \eta} B(\eta)
	+ \frac{3}{64\eta} B(\eta)^5 \bigg)
	\nonumber\\
	&= -1 - \frac{Q_-^2 \cos(\omega(\eta))^4}{\eta}
	+ O(\eta^{-3/2}),
	\label{eq:41}
	\end{align}
where $Q_- = \sqrt 2 E_B$ and we have used \eqref{eq:38}.
We remark that the assumption $B(\eta) = R(\eta) \cos(\omega(\eta))$
implies that $B'(\eta)$ is also computed as
	\begin{align}
	B'(\eta) = R'(\eta) \cos(\omega(\eta)) - R(\eta) \omega'(\eta) \sin(\omega(\eta)).
	\label{eq:42}
	\end{align}
Substituting \eqref{eq:40} and \eqref{eq:41} into \eqref{eq:42},
the identity $B'(\eta) = R(\eta) \sin(\omega(\eta))$ is realized.
Then from \eqref{eq:41},
for $\eta > 1$,
	\begin{align*}
	&\omega(\eta)\\
	&= \omega(1) - (\eta-1) - \int_1^\eta \frac{Q_-^2 \cos(\omega(\zeta))^4}{\zeta}
	+ O(\zeta^{-3/2}) d\zeta\\
	&= \omega_{\ast,-}
	- \eta - \int_1^\eta
	\frac{Q_-^2 (3 + \cos(4\omega(\zeta)) + 4 \cos(2 \omega(\zeta)))}{8\zeta} d \zeta
	+ O(\eta^{-1/2})\\
	&= \omega_{\ast,-} - \eta - \frac 3 8 Q_-^2 \log(\eta)
	- \int_1^\eta
	\frac{Q_-^2 (\cos(4\omega(\zeta)) + 4 \cos(2 \omega(\zeta)))}{8\zeta} d \zeta
	+ O(\eta^{-1/2}),
	\end{align*}
where $\omega_{\ast,-}$ is a constant.
Then, for the proof of  Proposition \ref{Proposition:1.4},
it is sufficient to show
	\[
	\int_1^\eta \frac{\cos(\omega(\zeta))}{\zeta} d \zeta
	= \mathrm{Const.} + O(\eta^{-1}).
	\]
Without loss of generality, we assume $\omega'(\eta) < - 1/2$ for $\eta > 1$.
If not, we can find a constant $\eta_0$
such that $\omega'(\eta) < - 1/2$ for $\eta > \eta_0$
and consider the integral on the interval between $\eta_0$ and $\eta$.
By integration by parts,
	\begin{align*}
	\int_1^\eta \frac{\cos(\omega(\zeta))}{\zeta} d \zeta
	&= \frac{\sin(\omega(\zeta))}{\zeta \omega'(\zeta)} \bigg|_1^\eta
	+ \int_1^\eta \frac{\omega'(\zeta) + \zeta \omega''(\zeta)}{\zeta^2 \omega'(\zeta)^2}
	\sin(\omega(\zeta)) d \zeta\\
	&= - \frac{\sin(\omega(1))}{\omega'(1)}
	+ \int_1^\infty \frac{\omega'(\zeta) + \zeta \omega''(\zeta)}{\zeta^2 \omega'(\zeta)^2}
	\sin(\omega(\zeta)) d \zeta
	+ O(\eta^{-1}),
	\end{align*}
where the second integral in the last line is bounded
because $\omega'$, $B$ and $B'$ are bounded.

In the case where $y > 1$,
\eqref{eq:22} is similarly obtained by replacing \eqref{eq:36} with
	\[
	B''(\eta) + B(\eta)
	=
	\frac 1 {2\eta} B(\eta)^3
	- \frac 3 {16 \eta^{2}} B(\eta)
	- \frac{3}{64 \eta^2} B(\eta)^5.
	\]
%]]]
%]]]

\appendix

%[[[ bib

%]]]
\end{document}